\documentclass[9pt]{article}
\usepackage{amsgen, amsmath, amsfonts, amsthm, amssymb,enumerate,amscd,graphics, amsfonts,dsfont,mdframed,comment}
\usepackage[all]{xy}
\newtheorem{defn}{Definition}[section]
\newtheorem{prop}[defn]{Proposition}
\newtheorem{lem}[defn]{Lemma}
\newtheorem{thm}[defn]{Theorem}
\newtheorem{cor}[defn]{Corollary}
\newtheorem{rem}[defn]{Remark}
\newtheorem {conj}[defn]{Conjecture}

\newcommand {\ZZ}{{\mathds Z}}
\newcommand {\NN}{{\mathds N}}

\newcommand {\C}{{\mathds C}}

\newcommand {\Z}{{\mathcal Z}}
\newcommand {\CC}{{\mathcal C}}
\newcommand {\Q}{{\mathds Q}}
\newcommand {\R}{{\mathds R}}

\newcommand {\M}{{\mathcal M}}

\newcommand {\BB}{{\mathbf B}}
\newcommand {\bff}{{\mathbf f}}
\newcommand {\BF}{{\mathbf F}}

\newcommand {\CP}{{\mathds P}}

\newcommand {\D}{{\mathcal D}}

\newcommand {\J}{{ J}}

\newcommand{\MM}{{\mathcal {MM}}}

\newcommand{\w}{{\omega}}

\def\Pic{\operatorname{Pic}}
\def\Ker{\operatorname{Ker}}

\def\div{\operatorname{div}}
\def\deg{\operatorname{deg}}
\def\mod{\operatorname{mod}}

\def\Im{\operatorname{Im}}
\def\ord{\operatorname{ord}}
\def\reg{\operatorname{reg}}
\def\Hom{\operatorname{Hom}}
\def\Ext{\operatorname{Ext}}
\def\Div{\operatorname{Div}}
\title{The Fundamental Group and Extensions of Motives of  Jacobians of Curves}
\author{Subham Sarkar and Ramesh Sreekantan
}

\begin{document}

\baselineskip=17pt
\maketitle

\begin{abstract}
In this paper we construct extensions of mixed Hodge structure coming from the mixed Hodge structure on the graded quotients of the group ring of the Fundamental group of a smooth projective pointed curve which correspond to the regulators of certain motivic cohomology cycles on the Jacobian of the curve essentially  constructed by Bloch and Beilinson. This leads to a new iterated integral expression for the regulator.  This is a generalisation of a theorem  of Colombo \cite{colo} where she constructed the extension corresponding to Collino's cycles in the Jacobian of a hyperelliptic curve. 


\end{abstract}


\section{Introduction}

A formula, usually called Beilinson's formula --- though independently due to  Deligne as well --- describes the motivic cohomology group of a smooth projective variety $X$  over a number field as the group of extensions in a conjectured category of mixed motives, $\MM_{\Q}$. If $i$ and $n$ are two integers then \cite{scho2},  
$$\Ext^1_{\MM_{\Q}}(\Q(-n),h^i(X))=\begin{cases} CH^n_{\hom}(X) \otimes \Q & \text{ if } i+1=2n\\ H^{i+1}_{\M}(X,\Q(n)) & \text{ if } i+1\neq 2n\end{cases}$$
Hence, if one had a way of constructing extensions in the category of mixed motives by some other method,  it would provide a way of constructing motivic cycles. 

One way of doing so is by considering the group ring of the fundamental group of the algebraic variety $\ZZ[\pi_1(X,P)]$. If $J_P$ is the augmentation ideal --- the kernel of the map from $\ZZ[\pi_1(X,P)]   \to \ZZ$ ---  then the graded pieces $J_P^a/J_P^b$ with $a<b$ are expected to have a motivic structure. These give rise to natural extensions of motives --- so one could hope that these extensions could be used to construct natural motivic cycles. 

Understanding the motivic structure on the fundamental group is complicated.  However, the Hodge structure on the fundamental group is well understood \cite{hain}.  The regulator of a motivic cohomology cycle can be thought of as the realisation of the  corresponding extension of motives as an extension in the category of mixed Hodge structures. So while we may not be able to construct motivic cycles as extensions of {\em motives} coming from the fundamental group - we can hope to construct their regulators as extensions of {\em mixed Hodge structures} (MHS)  coming from the fundamental group. 

The aim of this paper is to describe this  construction in the case of  the motivic cohomology group of the Jacobian of a curve. The first work in this direction is due to Harris \cite{harr} and Pulte \cite{pult}, \cite{hain}. They  showed that the Abel-Jacobi image of the modified diagonal cycle on the triple product of a pointed curve $(C,P)$, or alternatively the Ceresa cycle in the Jacobian  $J(C)$ of the curve, is the same as an extension class coming from $J_P/J_P^3$, where $J_P$ is the augmentation ideal in the group ring of the fundamental group of $C$ based at $P$.  

In \cite{colo}, Colombo extended this theorem to show that the regulator of a cycle in the motivic cohomology of a Jacobian of a {\em hyperelliptic} curve, discovered by Collino \cite{coll},  can be realised as an extension class coming from $J_P/J_P^4$, where here $J_P$ is the augmentation ideal of a  related curve.

In this paper we extend Colombo's result to more general curves. If $C$ is a smooth projective  curve of genus $g$  with a function $f$ with divisor $\div(f)=NQ-NR$ for some points $Q$ and $R$ and some integer $N$  and such that $f(P)=1$ for some other point $P$,  there is a motivic cohomology  cycle  $Z_{QR,P}$ in $H^{2g-1}_{\M}(J(C),\ZZ(g))$ discovered by Bloch \cite{blocirvine}.  We show that the regulator of this cycle can be expressed in terms of an extensions coming from $J_P/J_P^4$. When $C$ is hyperelliptic and $Q$ and $R$ are ramification points of the canonical map to $\CP^1$, this is Colombo's result.

A crucial step in Colombo's work is to use the fact that  the modified diagonal cycle is {\em torsion} in the Chow group $CH^2_{\hom}(C^3)$ when $C$ is a   hyperelliptic curve. This means the  extension coming from $J_P/J_P^3$ splits and hence does not depend on the base point $P$. This allows her to consider the extension for $J_P/J_P^4$. In general, that is {\em not} true --- in fact the known examples of non-torsion modified diagonal cycles come from the curves we consider - namely modular and Fermat curves. Our main contribution is to use an idea of Rabi \cite{rabi}  to show that Colombo's  arguments can be extended to work in our  case as well. As a result we have a more general situation --- which has some arithmetical applications.

\subsection{Main Theorem}

We have the following theorem ({\bf Theorem \eqref{mainthm}}):

 \begin{thm} Let $C$ be a smooth projective  curve of genus $g=g_C$ over a field $K$.  Let  $P$, $Q$ and $R$ be three  distinct $K$-rational points such that there is a function $f_{QR}$ with $\div(f_{QR})=NQ-NR$ for some $N$ and $f_{QR}(P)=1$. Let  $Z_{QR}=Z_{QR,P}$ be the element of the motivic cohomology group $H^{2g-1}_{\M}(J(C), \ZZ(g))$ constructed by Bloch. There exists an extension class  $\epsilon^4_{QR,P}$  in $\Ext^1_{MHS}(\ZZ(-2),\wedge^2 H^1(C))$ constructed from the mixed Hodge structures associated to  the fundamental groups $\pi_1(C\backslash Q,P)$ and $\pi_1(C \backslash R,P)$ such that 
 $$\epsilon^4_{QR,P}=(2g_C+1) \reg_{\ZZ}(Z_{QR}) $$
 in $\Ext^1_{MHS}(\ZZ(-2),\wedge^2 H^1(C ))$.  
 
 \end{thm}
 
In other words our theorem states that the regulator of a natural cycle in the motivic cohomology group of a product of curves, being thought of as an extension class,  is the same as  the extension class of a  natural extension of mixed Hodge structures coming from the fundamental group of the curve.  In fact, the extension is an  extension of {\em pure} Hodge structures. 

Our primary motivation are the conjectures  relating regulators of the motivic cycles to special values of $L$-functions.  One application we have is to the case of modular curves. Beilinson \cite{beil} constructed a cycle in the group $H^3_{\M}(X_0(N) \times X_0(N), \Q(2))$ and showed that its regulator  is related to a special value of the $L$-function. We construct the extension of $MHS$ coming from the fundamental group which corresponds to the regulator of the image of this cycle in the Jacobian of $X_0(N)$. Or equivalently, this is the projection on to the sub-motive $\wedge^2 H^1(C)$  of $\otimes^2 H^1(C)$. 

Since the mixed Hodge structure associated to the fundamental group is related to iterated integrals we also get an expression for the regulator as an iterated integral. In a subsequent paper we apply this in the case of  Fermat curves to get an explicit expression for the regulator in terms of hypergeometric functions --- analogous to the works of Otsubo \cite{otsu1},\cite{otsu2}.

Darmon-Rotger-Sols  \cite{DRS} have used the modified diagonal cycle to construct points on Jacobians of the curves and used the iterated integral approach to find a formula for the Abel-Jacobi image of these points. Starting with Bloch \cite{blocirvine} and later Collino \cite{coll} and Colombo \cite{colo} it has been known that these null homologous cycles degenerate to higher Chow cycles on related varieties. Recently Iyer and M\"uller-Stach \cite{iymu} have shown that the modified diagonal cycle degenerates to the kind of cycles we consider in some special cases. This degeneration can be understood from the point of view of extensions and we make a few remarks on that.

{\em Acknowledgements:} This work constitutes part of the PhD. thesis of the first author. We would like to thank  Najmuddin Fakhruddin, Noriyuki Otsubo, Satoshi  Kondo, Elisabetta  Colombo,  Jishnu Biswas,  Manish Kumar, Ronnie Sebastian, Ranier Kaenders and Suresh Nayak for their comments and suggestions. Finally, it gives us great  pleasure to thank the Indian Statistical Institute, Bangalore  for their support while this work was done. 

\section{Iterated Integrals, Cycles, Extensions and Regulators}

\subsection{Iterated Integrals}

Let $\alpha:[0,1] \to X$ be a path and $\omega_1,\omega_2,\dots,\omega_n$ be $1$-forms on $X$. Suppose $\alpha^*(\omega_i)=f_i(t) dt$. The  {\em iterated integral of length n} of $\w_1,\w_2 \dots \w_n$ is defined to be 
$$\int_{\alpha} \omega_1\omega_2\dots \omega_n:=\int_{0 \leq t_1 \leq t_2\leq \dots \leq t_n \leq 1} f_1(t_1)f_2(t_2)\dots f_n(t_n)dt_1dt_2 \dots  dt_n.$$
An {\em iterated integral of length $\leq n$} is a linear combination of integrals of the form above with lengths $\leq n$. It is said to be a {\em homotopy functional} if it only depends on the homotopy class of the path $\alpha$.  A homotopy functional gives a functional on the group ring of the fundamental group or path space. 

Iterated integrals can be thought of as integrals on simplices and satisfy the following basic properties --- here we have only stated the results for length two iterated integrals, since that is the only type we will encounter in this paper.

\begin{lem}[Basic Properties] Let $\omega_1$ and $\omega_2$ be smooth $1$-forms on $C$ and $\alpha$ and $\beta$ piecewise smooth paths on $C$ with $\alpha(1)=\beta(0)$. Then
\begin{enumerate} 
\item $\displaystyle{\int_{\alpha \cdot \beta} \omega_1 \omega_2=\int_{\alpha} \omega_1\omega_2 + \int_{\beta} \omega_1 \omega_2 + \int_{\alpha} \omega_1 \int_{\beta} \omega_2}$
\item $\displaystyle{\int_{\alpha} \omega_1 \omega_2 + \int_{\alpha} \omega_2 \omega_1 =  \int_{\alpha} \omega_1 \int_{\alpha} \omega_2}$
\item $ \displaystyle{\int_{\alpha}  df \omega_1 = \int_{\alpha} f \omega_1 - f(\alpha(0)) \int_{\alpha} \omega_1}$
\item $\displaystyle{\int_{\alpha} \omega_1 df = f(\alpha(1))\int_{\alpha} \omega_1 -\int_{\alpha} f\omega_1}$
\end{enumerate}
\label{basicproperties}

\end{lem}

\begin{proof} This can be found in any article on iterated integrals, for instance Hain's excellent article\cite{hain}.\end{proof}

\subsection{Motivic Cohomology Cycles}
\label{motiviccycles}

Let $X$ be a smooth projective algebraic variety of dimension $g$  defined over a number field $K$. An element of the motivic cohomology group $H^{2g-1}_{\M}(X,\ZZ(g))$  has the following presentation: Elements  are represented by finite sums 
$$Z=\sum_i (C_i,f_i)$$
where $C_i$ are curves on $X$ and $f_i:C_i \longrightarrow \CP^1$ are functions on them subject to the co-cycle condition 
$$\sum_i \div(f_i)=0.$$
The relations in this group are given by the tame symbols of functions $\{f,g\}$ in $K_2(K(X))$, where $K(X)$ is the function field,  
$$\tau(\{f,g\})=\sum_{W\in X^{(1)}} (-1)^{\ord_W(f)\ord_W(g)}\frac{f^{\ord_W(g)}}{g^{\ord_W(f)}}.$$
where $X^{(1)}$ is the  collection of  codimensional $g-1$ cycles on $X$. The group $H^{2g-1}_{\M}(X,\ZZ(g)) \otimes \Q$ is the same as the higher Chow group $CH^g(X,1) \otimes \Q$.

If $L/K$ is a finite extension, let $X_L=X \otimes_{K} L$. There is a norm map 
$$Nm_{L/K}: H^{2g-1}_{\M}(X_L,\ZZ(g)) \longrightarrow H^{2g-1}_{\M}(X,\ZZ(g)).$$
In the group $H^{2g-1}_{\M}(X,\ZZ(g))$ there are certain {\em decomposable} cycles coming from the product 
$$H^{2g-1}_{\M}(X,\ZZ(g))_{dec}=\bigoplus_{L/K\; finite} Nm_{L/K} \left( \Im \left(H^1_{\M}(X_L,\ZZ(1)) \otimes H^{2g-2}_{\M}(X_L,\ZZ(g-1)) \longrightarrow H^{2g-1}_{\M}(X_L,\ZZ(g))\right)\right).$$ 
The group of  {\em indecomposable} cycles  is defined as the quotient ---
$$H^{2g-1}_{\M}(X,\ZZ(g))_{ind}=H^{2g-1}_{\M}(X,\ZZ(g))/H^{2g-1}_{\M}(X,\ZZ(g))_{dec}.$$
In general it is not easy to find non trivial  elements in this group. One of the aims of this paper is to show that in certain cases the cycles we construct are indecomposable. One way to do that is by computing its regulator.

\subsection{Regulators}

Let $X$ be a smooth projective algebraic variety of dimension $g$  over $\C$. The regulator map of Beilinson is a  map from the motivic cohomology group to the Deligne cohomology group.
$$\reg_{\ZZ}: H^{2g-1}_{\M}(X, \ZZ(g)) \rightarrow H^{2g-1}_{\D}(X, \ZZ(g))= \frac {(F^1H^2(X,\C))^*}{H_2(X,\ZZ(2))}.$$
where $*$ denotes the $\C$-linear dual and $F^{\bullet}$ denotes the Hodge filtration. The Deligne cohomology group $H^3_{\D}(X,\ZZ(2))$ is a generalised torus. 

The map is defined as follows \cite{coll}: Let   $Z=\sum_i (C_i,f_i)$ be a cycle   in $H^{2g-1}_{\M} (X,\ZZ(g))$, so $C_i$ are curves on $X$ and $f_i$ are functions on them satisfying the cocycle condition. Let $[0,\infty]$ denote the positive real axis in $\CP^1$ and $\gamma_i=f_i^{-1}([0,\infty])$. Then $\sum_i \div(f_i)=0$ implies that the $1$-chain $\sum_i \gamma_i$ is closed and  in fact torsion.  If $H_1(X,\ZZ)$ has no torsion -- as is in the case of a product of curves, it is exact.  Assuming that, we have  
$$\sum_i \gamma_i=\partial(D)$$
for some  $2$-chain $D$. For  $\omega$  a closed $2$-form whose cohomology class lies in $F^1H_{DR}^2(X,\C)$,
\begin{equation}
\reg_{\ZZ}(Z)(\omega):=\sum_i \int_{C_i - \gamma_i } \log(f_i) \omega + 2\pi i  \int_D \omega
\label{regulatorformula}
\end{equation}
For a decomposable element $(C,a)$, where $a \in \C^*$,  the regulator is particularly simple:
$$\reg_{\ZZ}((C,a))(\omega)=\int_{C} \log(a)\omega=\log(a) \int_C \omega.$$

\subsection{Extensions}

As stated in the introduction, conjecturally, there is a canonical  description of the motivic cohomology group as an extension in the category of mixed motives. From now on, $\Ext$ will denote $\Ext^1$. Further, we will use $H^*(X)$ to denote the group $H^*(X(\C),\ZZ)$, the singular (Betti) cohomology group with integral coefficients  and $H^*(X)_A$ to denote $H^*(X)\otimes_{\ZZ} A$,  where $A$ is typically $\Q, \R$ or $\C$. 

In our case, if one has a suitable category of mixed motives over $\Q$,  $\MM_{\Q}$,  one expects  for a variety $X$ \cite{scho2} and $i,n$ non-negative numbers  with $i<2n-1$, 
\begin{equation}
H^{i+1}_{\M} (X,\Q(n))\simeq \Ext_{\MM_{\Q}} (\Q(-n),h^i(X)). 
\label{extmot}
\end{equation}
where $\Q(-n)$ denotes the twist of the Tate motive and $h^i(X)$ denotes the motive whose Hodge realisation is $H^i(X)$. 

One {\em knows} that the Deligne cohomology can be considered as an extension in the category of integral mixed Hodge structures, 
$$H^{i+1}_{\D} (X,\ZZ(n))\simeq \Ext_{MHS} (\ZZ(-n),H^i(X)) $$

Assuming \eqref{extmot} holds at the level of integer coefficients,  the regulator map  above then has a canonical description as the map induced by the realisation map from the category of mixed motives to the  category of mixed Hodge structures, 
$$\Ext_{\MM_{\Q}} (\ZZ(-n),h^i(X)) \stackrel{\reg_{\ZZ}}{\longrightarrow} \Ext_{MHS}(\ZZ(-n),H^i(X)).$$
One can take a further realisation in the category of Real mixed Hodge structures $\R-MHS$ to get the real regulator map to Deligne cohomology with $\R$-coefficients 
$$\reg_{\R}:H^{i+1}_{\M}(X,\ZZ(n)) \longrightarrow H^{i+1}_{\D}(X_{\C},\R(n))=\Ext_{\R-MHS}(\R(-n),H^i(X)_{\R}).$$
In our case of $X=J(C)$ a a jacobian of a genus $g$ curve which is a variety of dimension $g$  and $i=2g-2$ and $n=g$, 
$$H^{2g-1}_{\D}(X_{\C},\R(g)) \simeq \Ext_{MHS}(\ZZ(-g),H^{2g-2}(X))$$
$$\simeq \Ext_{MHS}(\ZZ(-2),H^2(X))\simeq H^2_B(X_{\C},\R(1)) \cap H^{1,1}(X)$$
and the real regulator of a cycle  can be viewed as a current on $(1,1)$-forms. 

If $X$ is a variety defined over $\Q$ one can take a further realization in to the category of $\R$-Hodge structures with an action of the Frobenius at $\infty$. This gives the Beilinson regulator map to `real' Deligne cohomology which is used in the conjectures on special values of $L$-functions \cite{beil}.

\subsection{Extensions of mixed Hodge structures coming from the Fundamental Group}

The key point of this paper is that,  in some cases, one can also obtain extensions of mixed Hodge structures in other ways. For instance, if $(X,P)$ is a pointed algebraic variety, it was shown by Hain \cite{hain}  that the graded quotients $J_P^a/J_P^b$, with $a\leq b$,  where $J_P$ is the augmentation ideal of the  group ring of the fundamental group $\ZZ[\pi_1(X,P)]$, carry mixed Hodge structures. Hence natural exact sequences involving them lead to extensions of mixed Hodge structures. 

Our aim is to first construct some natural motivic cohomology cycles in the case when $X=J(C)$, the Jacobian of a curve of genus $g$. Their regulators will give rise to extensions of mixed Hodge structures.  We will show that there are natural extensions of mixed Hodge structures coming from the Hodge structure on the graded pieces of  $\ZZ[\pi_1(C,P)]$  for some suitable point $P$  which give the {\em same}  extensions.  In particular, since the constructions can be carried out in at the level of mixed motives, if we had a good category of mixed motives  the  cycle {\em itself} would be an extension in the conjectured category of mixed motives coming from the fundamental group. 

\section{A Motivic Cohomology Cycle on $J(C)$}

In this section we construct a motivic cohomology cycle on $J(C)$, where $C$ is a smooth projective curve over a number field $K$. This was  first constructed by Bloch \cite{blocirvine} in the case when $C$ is the modular curve $X_0(37)$. The cycle is similar, in fact, generalises, the cycle constructed by Collino \cite{coll}. This section generalises the work of Colombo \cite{colo} on constructing the extension corresponding to the Collino cycle and hence many of the arguments are adapted from her paper. 

\subsection{The cycle $Z_{QR,P}$}

Let $C$ be a smooth projective curve defined over a number field $K$. Let $Q$ and $R$ be two distinct $K$-rational points on $C$ such that there is a function  $f=f_{QR}$ with divisor 
$$\div(f_{QR})=NQ-NR$$
for some $N \in \NN$.  To determine the function precisely, we choose a distinct third point $P$ and assume $f_{QR}(P)=1$. 

There exist notable examples of curves where such functions can easily be found. For instance, {\em modular curves} with $Q$ and $R$ being cusps, {\em Fermat curves} with the two points being among  the `trivial' solutions of Fermat's Last Theorem, namely the points with one of the coordinates being $0$, and  {\em hyperelliptic curves} with the two points being Weierstrass points. 

Let $C_Q$ denote the image of $C$ under the map $C \rightarrow J(C)$ given by $x \rightarrow x-Q$. Similarly, let $_RC$ denote the image of $C$ under the map $x \rightarrow R-x$ and let $f_Q$ and $_Rf$ denote the function $f$ being considered as a function on $C_Q$ and $_RC$ respectively. 

Consider the cycle in $J(C)$ given by 
$$Z_{QR,P}=(C_Q,f_Q)+(_RC,_Rf)$$ 
$$\div_{C_Q} (f_Q) + \div_{_RC}(_Rf)=N(0)-N(R-Q) + N(R-Q)-N(0)=0.$$
Hence the cycle $Z_{QR,P}$ gives an element of $H^{2g-1}_{\M}(J(C),\ZZ(g))$.

This  cycle was first described by Bloch \cite{blocirvine} in his celebrated Irvine lecture notes and later variants of this construction were used by Beilinson and others to verify the Beilinson conjectures in some special cases. They defined a  cycle on $C \times C$ but under the natural map 
 $$C \times C \longrightarrow J(C)$$
 $$(x,y) \longrightarrow (x-y)$$
 their cycle maps to $Z_{QR,P}$. When $C$ is hyperelliptic one has a function with divisor $\div(f)=2(Q)-2(R)$ where $Q$ and $R$ are ramification points and in this case the cycle was considered by Collino \cite{coll}. 

\subsection{The Regulator of $Z_{QR,P}$} 

Let $Z_{QR,P}$ be the motivic cohomology cycle in $H^{2g-1}_{\M}(J(C), \ZZ(g))$.  We now obtain a formula for its regulator.  The regulator is a current on forms in $F^1(H^2(J(C)_{\C}))$. Since $H^2(J(C))=\wedge^2 H^1(C)$ elements are of the form $\phi \wedge \psi$ where $\phi$ and $\psi$ are closed $1$-forms on $C$ and one of $\phi$ or $\psi$ is of type $(1,0)$.

We have the following theorem: 
 \begin{thm} Let $Z_{QR,P}$ be the motivic cohomology cycle in  $H^{2g-1}_{\M}(J(C),\ZZ(g))$ and $\phi$ and $\psi$ two closed harmonic  $1$-forms in  $H^1(J(C))=H^1(C)$ with $\psi$ holomorphic. Then 
 \begin{align*}
\reg_{\ZZ}(Z_{QR,P})( \phi \wedge \psi) &= 2 \int_{C-\gamma} \log(f)\phi \wedge \psi   + 2\pi i  \int_{\gamma} (\phi\psi -  \psi\phi )\\
&= 2\left(\int_{C-\gamma} \log(f)\phi \wedge \psi   + 2\pi i  \int_{\gamma} \phi\psi  \right)
\end{align*}
\label{regform}
\end{thm}

\begin{proof}
The proof is a consequence of the following lemmas. Recall that $f=f_{QR}$ is a function on $C$ with divisor $NQ-NR$ for some $N$. Let  $\omega=\phi \wedge \psi$ and $\gamma=f^{-1}([0,\infty])$. As $f$ is of degree $N$, $\gamma$ is the union of $N$ paths --- each lying on a different sheet with only the points $Q$ and $R$ in common. We will denote them by $\gamma^i$, $1\leq i \leq N$. Each $\gamma^i$ is a path from $Q$ to $R$.  Let $\gamma_Q$ and $_R\gamma$ denote the path $\gamma$ on $C_Q$, $_RC$ respectively and similarly for the components $\gamma^i$. Then from the co-cycle condition one has 
$$  \gamma_Q \cdot _R\gamma^-=\partial(D)$$
where $D$ is a $2$-chain on $J(C)$. Here for a path $\alpha$, $\alpha^-$ is the inverse: $\alpha^-(t)=\alpha(1-t)$. 

From equation \eqref{regulatorformula}  one has 
\begin{equation}
\reg_{\ZZ}(Z_{QR})(\omega )=\int_{C_Q} \log(f_Q ) \omega + \int_{_RC} \log(_Rf) \omega+ 2\pi i \int_{D} \omega.
\label{regfor1}
\end{equation}
Our aim is to find a more explicit expression for $\reg_{\ZZ}(Z_{QR})$. For this we need an explicit description of $D$. This was done by Colombo \cite{colo}, Lemma 1.2.
\begin{lem} Let 
\[
 a(s,t)=t \hspace{2cm} {\rm and }  \hspace{2cm}  b(s,t)=\frac{t(1-s)}{1-s(1-t)}. \]
Define $F_i:[0,1] \times [0,1] \longrightarrow J(C)$ by 
$$F_i(s,t)=\gamma^i(a(s,t)) - \gamma^i(b(s,t))$$
for $1 \leq i \leq N$ and let 
$$D_i=\Im(F_i).$$
Then 
\[
\partial(D_i)= \gamma_Q^i \cdot  _R\gamma^{i-}.
\]
In particular,  if $D=\cup_{i=1}^{N} D_i$ then 
\[
\partial(D)=  \gamma_Q \cdot _R\gamma^- 
\]
\end{lem}

\begin{proof} The proof is essentially identical to Colombo's Lemma 1.2 - the only change is that she does it for $N=2$ - so we do not repeat it here.

%
%
%

\end{proof}
We can compute the last integral as an iterated integral as follows. 

\begin{lem} Let $\phi$ and $\psi$ be closed $1$-forms on $C$  and let $D_i$ be a disc as   in the above lemma. Then 
$$\int_{D_i} \phi \wedge \psi = \int _{\gamma^{i}_Q}  \phi \psi - \int_{_R\gamma^{i-}} \psi \phi$$
\label{secondintegral}
\end{lem}

\begin{proof} This again is a slightly modified version of Colombo \cite{colo}, Lemma 13. 

\end{proof}

To complete the proof we combine  the above lemma with the earlier expression for the regulator and then have to pull back to $C$. Pulling back from $C_Q$ to  $C$ involves  a translation and pulling back  from $_RC$ is a combination of  $(-1)^*$ followed by a translation. $(-1)^*$ preserves $\psi \phi$ and $\phi \wedge \psi$   and since the forms are harmonic, they are translation invariant as well. Finally, since $\gamma=\partial D$ is exact and $\phi$ and $\psi$ are closed, one has, from Lemma \ref{basicproperties} (2) and Stokes theorem
$$\int_{\gamma} \phi\psi + \int_{\gamma} \psi\phi = \int_{\gamma} \phi \int_{\gamma} \psi =0$$
%

\end{proof}

\section{The Fundamental group and  Mixed Hodge Structures.} 

Let $C$ be a smooth projective  curve and $P$, $Q$ and $R$ be three distinct  points on $C$. Consider the open curve $C_Q=C \backslash \{Q\}$. Let $\ZZ[\pi_1(C_Q,P)]$ be the group ring of the fundamental group of $C_Q$ based at $P$.  Let $J_{Q,P}:=J_{C_Q,P}$ denote the augmentation ideal --- 
$$J_{Q,P}:=J_{C_Q,P}=\Ker \{ \ZZ[\pi_1(C_Q,P)] \stackrel{\deg}{\longrightarrow} \ZZ \}.$$
Let $H^0({\mathcal B}_r(C_Q;P))$ denote the $F$-vector space, where $F$ is $\R$ or $\C$, of homotopy invariant iterated integrals of length $\leq r$.  Chen\cite{chen} showed that    
$$H^0({\mathcal B}_r(C_Q;P)) \simeq  \Hom_{\ZZ}(\ZZ[\pi_1(C_Q,P)]/J_{C_Q.P}^{r+1},F)$$
under the map 
$$ I \longrightarrow   I(\gamma)=\int_{\gamma} I .$$
Using this Hain \cite{hain} was able to put  a natural  mixed Hodge structure on the graded pieces $J_{Q,P}/J_{Q,P}^r$. 

\subsection {Extensions.} 
From this point on, we will use the following notation. For an extension $E$ of mixed Hodge structures, 
$$E:0 \longrightarrow  B \longrightarrow H \longrightarrow A \longrightarrow 0$$
we use $m$ to denote its class in $\Ext^1_{MHS}(A,B)$ and $H$ to denote the middle term. Unless otherwise stated  $\Ext$ will denote $\Ext^1_{MHS}$, the first extension group in the category of mixed Hodge structures. We will also use the notation $N\cdot E$ to denote $N$ times the extension with respect to the Baer sum, use $N \cdot m$ to denote its class of this extension in the $\Ext$ group and $N\cdot H$ to denote its middle term . 

\subsection{The extension $E^3_{Q,P}$.}

One can consider the extensions of mixed Hodge structures 
$$E^r_{Q,P}: 0 \longrightarrow (J_{Q,P}/J_{Q,P}^{r-1})^* \longrightarrow (J_{Q,P}/J_{Q,P}^{r})^*\longrightarrow (J_{Q,P}^{r-1}/J_{Q,P}^{r})^*\longrightarrow 0$$
where for a module $M$, $M^*=\Hom(M,\ZZ)$. 

The simplest non-trivial case is when $r=3$.  In this case $(J_{Q,P}/J_{Q,P}^2)^* \simeq H^1(C_{Q})\simeq H^1(C)$ and  $(J_{Q,P}^2/J_{Q,P}^3)^* \simeq  \otimes^2 H^1(C)$ and the exact sequence becomes 
$$E^3_{Q,P}: 0 \longrightarrow H^1(C) \longrightarrow (J_{Q,P}/J_{Q,P}^3)^* \longrightarrow \otimes^2 H^1(C) \longrightarrow 0.$$
Hence $E^3_{Q,P}$ gives an element $m^3_{Q,P}$  in $\Ext(\otimes^2 H^1(C),H^1(C))$. A similar construction with $R$ in the place of $Q$ gives us the extension $E^3_{R,P}$, which also lies in the same $\Ext$ group. 

There is  a surjection $\cup:\otimes^2 H^1(C) \longrightarrow  H^2(C)  \simeq \ZZ(-1)$ coming from the cup product. Let $K$ be the kernel of this map. The exact sequence  of Hodge structures 
$$0 \longrightarrow K \longrightarrow \otimes^2 H^1(C) \stackrel{\cup}{\longrightarrow} \ZZ(-1)  \longrightarrow 0$$
splits over $\Q$  but {\em not} over $\ZZ$.  This happens as follows: There is a bilinear form \cite{kaen}
$$b:\otimes^2 H^1(C) \times \otimes^2 H^1(C) \longrightarrow \ZZ$$
defined by 
$$b(x_1 \otimes x_2,y_1 \otimes y_2)=(x_1 \cup y_2) \cdot (x_2 \cup y_1).$$
Let $S$ denote the orthogonal complement of $K$ in $\otimes^2 H^1(C)$ with respect to this bilinear form. Then, under the cup product  $S$ projects to  $2g_C \ZZ(-1)$ where $g_C$ is the genus of $C$  and 
$$\otimes^2 H^1(C)_{\Q}=K_{\Q} \oplus S_{\Q}.$$
Let $\bar{m}^3_{Q,P}$ denote the class  in $\Ext_{MHS}(S,H^1(C))$ corresponding to the extension 
$$0 \longrightarrow H^1(C) \longrightarrow \bar{E}^3_{Q,P} \longrightarrow S \longrightarrow 0$$
obtained by restricting $E^3_{Q,P}$ to the extension of $S$ by $H^1(C)$. From Kaenders \cite{kaen} one knows there is a covering map of complex tori, 
$$\Ext(\otimes^2 H^1(C),H^1(C)) \stackrel{\phi}{\longrightarrow} \Ext(K \oplus S, H^1(C))=\Ext(K,H^1(C)) \times \Ext(S,H^1(C)).$$
It is well known that $\Ext(S,H^1(C))=\Ext(\ZZ(-1),H^1(C))\simeq \Pic^0(C)$.  To understand the other term, from the work of
 Hain \cite{hain}, Pulte \cite{pult}, Kaenders \cite{kaen} and Rabi \cite{rabi} one has the following theorem 
\begin{thm} The image of the  class  $m^3_{Q,P}$ of $E^3_{Q,P}$ in $\Ext(\otimes^2 H^1(C) ,H^1(C))$  is given by 
 $$\phi(m^3_{Q,P})=(m^3_P, \bar{m}^3_{Q,P})$$ 
 where $m^3_P \in \Ext(K,H^1(C))$ depends only on $P$ and $\bar{m}^3_{Q,P}$ is given by 
 
$$2 g_C Q  - 2P -\kappa_C \in \Pic^0(C)$$
where $\kappa_C$ is the canonical divisor of  $C$ and $g_C$ is the genus of $C$. 
\label{moddiag}
\end{thm}

Recall that in the group $\Ext$,  addition is given by the Baer sum. We will denote this by $\oplus_B$  (or $\ominus_B$ if we are taking differences). Let $m^3_{QR,P}$ denote the Baer difference $m^3_{Q,P} \ominus_{B} m^3_{R,P}$.
\begin{lem} Under the hypothesis that there is a function with divisor $\div(f_{QR})=NQ-NR$  the extension class $m^3_{QR,P}$ is torsion in $\Ext(H^1C) \otimes H^1(C),H^1(C))$.  Precisely, 
$$N\cdot H^3_{QR,P} \simeq H^1(C) \bigoplus \otimes^2 H^1(C)$$
where by $N\cdot H^3_{QR,P}$ we mean the middle term of the exact sequence obtained by adding the sequence $E^3_{QR,P}$  to itself $N$-times using the Baer sum. 
\label{splitting}
\end{lem}

\begin{proof} This follows from \cite{kaen}[Theorem 2.5] which states that the map 
$$\Pic^0(C) \longrightarrow \Ext( H^1(C) \otimes H^1(C),H^1(C))$$ 
given by 
$$Q-R \longrightarrow m^3_{Q,P}-m^3_{R,P}$$
is well defined and  injective. Hence, since $N(Q-R)=0$ in $\Pic^0(C)$, $N(m^3_{Q,P}-m^3_{R,P})=N(m^3_{QR,P})=0$ in $\Ext( H^1(C) \otimes H^1(C),H^1(C))$. 
\end{proof}

A consequence of this is that there is a morphism of integral mixed Hodge structures 
$$r_{3}: N \cdot H^3_{QR,P} \longrightarrow  H^1(C)$$
given by the projection.

\begin{rem}
This extension represents the class $Q-R$, at least up to a integral multiple, and is hence the first example of the theme of this paper - namely the Abel-Jacobi image of a  null-homologous cycle is described in terms of  extensions coming from the fundamental group. 
\end{rem}

\subsection{The extensions $E^4_{Q,P}$ and $E^4_{R,P}$}

From the work of Hain, Pulte, Harris and others one knows that the class $m^3_P$ in $\Ext(K,H^1(C))$ corresponds to the extension of mixed Hodge structures determined by the Ceresa cycle in $J(C)$, or alternately,  the modified diagonal cycle in $C^3$. 

We would like to construct a similar class corresponding to the motivic cohomology cycle $Z_{QR}$.  To that end,  we now consider, with $C, P,Q$ and $R$ as before, the extension corresponding to $r=4$
$$E_{Q,P}^4: 0 \longrightarrow (J_{Q,P}/J_{Q,P}^3)^*  \longrightarrow (J_{Q,P}/J_{Q,P}^4)^* \longrightarrow (J_{Q,P}^3/J_{Q,P}^4)^* \longrightarrow 0.$$
We have that $(J_{Q,P}^3/J_{Q,P}^4)^*\simeq \otimes^3 H^1(C)$ and this does not depend on $P, Q$ or $R$. However, from Theorem \ref{moddiag},  $(J_{Q.P}/J_{Q,P}^3)^*$ depends on $Q$ and $P$. Similarly  $(J_{R.P}/J_{R,P}^3)^*$ depends on $R$ and $P$.  Hence we get classes in $\Ext(\otimes^3 H^1(C),(J_{Q.P}/J_{Q,P}^3)^*)$ and $\Ext(\otimes^3 H^1(C),(J_{R.P}/J_{R,P}^3)^*)$ -- which are different groups - hence we cannot take their difference. 

When $C$ is hyperelliptic  the extension classes $m^3_{Q,P}$ and $m^3_{R,P}$  are 2-torsion in  $\Ext(\otimes^2 H^1(C),H^1(C))$. Hence one gets two classes 
$$2m^4_{Q,P},  2m^4_{R,P}  \in \Ext(\otimes^3 H^1(C),\otimes^2 H^1(C) \oplus H^1(C))$$
and one can project to get two classes $e^4_{Q,P}$ and $e^4_{R,P}$ in $\Ext(\otimes^3 H^1(C), H^1(C))$.  Colombo \cite{colo} shows that the class 
$$e^4_{QR,P}=e^4_{Q,P} \ominus_B e^4_{R,P} \in \Ext(\otimes^3 H^1(C), H^1(C))$$
corresponds to the extension determined by  the  cycle  $Z_{QR}$ --- after pulling back and pushing forward with some standard maps.

Unfortunately, in general the extension classes  $m^3_{Q,P}$ and $m^3_{R,P}$ are {\em not} torsion in the $\Ext$ group. They correspond to the instances where the Ceresa cycle is non-torsion --- which is the generic case. In fact,  the instances where it is  known that the cycles are non-torsion are precisely the cases we have in mind --- modular curves and Fermat curves \cite{harr},\cite{blocspec}. Hence we cannot use this argument immediately. However, since we know from Lemma \ref{splitting} that their difference $m^3_{QR,P}$ is torsion, we would like to get an extension of the form  
$$0 \longrightarrow H^3_{QR,P}\otimes \Q \longrightarrow ``H^4_{QR,P}" \otimes \Q \longrightarrow \otimes^3 H^1(C)\otimes \Q \longrightarrow 0$$
where $``H^4_{QR,P}"$ is the middle term of a sort of  generalised Baer difference of the two extensions $E^4_{Q,P}$ and $E^4_{R,P}$. We could then  push-forward this extension using the splitting to get a class in $\Ext(\otimes^3 H^1(C)_{\Q},H^1(C)_{\Q})$. We cannot simply consider $E^4_{QR,P}=E^4_{Q,P} \ominus_{B} E^4_{R,P}$ as the two extensions lie in different $\Ext$ groups.   So we have to consider a generalisation of Baer sums to not necessarily exact sequences which we came across in a paper of Rabi \cite{rabi}.
 
 
\subsection{The Baer sum}

This is well known but we recall it to fix notation in order to describe Rabi's work. Recall that if  we have two exact sequences of modules 
$$\begin{CD}
E_j: 0 @>>> A @>f_j>> B_j @>p_j>> C @>>>0 
\end{CD}$$
for $j \in \{1,2\}$, the Baer difference  $E_1 \ominus_B E_2$  is constructed as follows. We have 
$$\begin{CD}
0 @>>> A \oplus A @>f_1\oplus f_2>> B_1 \oplus B_2 @>p_1\oplus p_2>> C\oplus C @>>> 0.
\end{CD}
$$  
Let  $\psi: B_1\oplus B_2 \longrightarrow C$ be the map 
$$\psi(b_1,b_2)=p_1(b_1)-p_2(b_2)$$
and let  $H=\Ker(\psi)=\{(b_1,b_2)| \; p_1(b_1)=p_2(b_2)\}$.  Let $D$ be the image of $\tilde{f}: A \longrightarrow A \oplus A  \longrightarrow H$
$$\tilde{f}(a)=(f_1(a),f_2(a))$$ 
Let $B=H/D$. The map $f: A \oplus A \longrightarrow B$ given by 
$$f(a_1,a_2)=(f_1(a_1),f_2(a_2))$$
factors through $(A \oplus A)/A \simeq A$ and so one has a map $\bar{f}:A \longrightarrow B$, 
$$a \longrightarrow (f_1(a),0)=(0,-f_2(a))$$ 
and an exact sequence 
$$\begin{CD}
0 @>>> A @>\bar{f}>> B@>p_1( or \;p_2) >> C @>>> 0
\end{CD}
$$
The class of this exact sequence  in $\Ext(C,A)$ is the Baer difference  $E_1 \ominus_B E_2$.  The Baer sum $E_1 \oplus_B E_2$  is the sequence obtained when one of the maps $f_2$ or $p_2$ is replaced by its negative.  The Baer sum is essentially the push-out over $A$ in the category of modules.

\subsection{Rabi's  generalisation} Now suppose we  have diagrams of the  following type:
\[
\begin{CD}
 &&0\\&&@VVV\\
 &&A_1\\
 &&@VVi_jV\\
 0 @>>> B_1^j @>f_j>> B_2^j @>p_j>> B_3 @>>>0  \\
&&@VV\pi_j V \\
&&C_1\\
&&@VVV\\
&&0
\end{CD}
\]
 where the vertical and horizontal sequences are exact for  $j\in \{1,2\}$.   Let $E_j$ denote the horizontal exact sequences:
$$\begin{CD}
E_j: 0 @>>> B_1^j @>f_j>> B_2^j @>p_j>> B_3 @>>>0.
\end{CD}$$
 We  would like to take the Baer  difference of the $E_j$ --- but since they do not lie in the same $\Ext$ group we cannot quite do that. However, we can still salvage something. 
 
One gets two types of extension classes in $\Ext$ groups which do not depend on $j$. The vertical exact sequences give  classes  in $\Ext(C_1,A_1)$. We can form their Baer difference to get an exact sequence 
$$
\begin{CD} 
0 @>>> A_1 @>>> \BB_1 @>>> C_1 @>>>0.
\end{CD}
$$
The horizontal exact sequences give extensions in $\Ext(B_3,B_1^j)$. These depend on $j$ but their push forward under  $\pi_j$ give classes ${\mathbf f}_{B_2^j}$ in  $\Ext(B_3,C_1)$.  

Define $\BB_2$ as follows: Let $H_2=\Ker(\psi)$, where $\psi$ is the `difference'  map 
 $$ \psi: B_2^1\oplus B_2^2 \longrightarrow B_3 $$
 $$ \psi((b_2^1,b_2^2)) = (p_1(b_2^1)-p_2(b_2^2))$$
 Let $D_2$ be the image of the map 
 $$A_1 \longrightarrow B_1^1 \oplus B_1^2 \longrightarrow H_2$$
 $$ a \longrightarrow (f_1(i_1(a)),f_2(i_2(a)))$$
Define $\BB_2=H_2/D_2$. We call this the {\em generalised} Baer difference of $E_1$ and $E_2$ and denote it by $\tilde{\ominus}_B$.  Observe that this is almost the Baer difference  of $E_1$ and $E_2$ in the sense that if $B_1=B_1^1=B_1^2$, then we could take the difference  in $\Ext(B_3,B_1)$. Since that is not the case, we do the best we can ---  we take the difference  of the {\em inexact} sequences 
 $$
 \begin{CD}
 0 @>>>A_1 @>>> B_2^j @>>>B_3 @>>> 0.
 \end{CD}
 $$
 As a result of this one has a complex
 $$
 \begin{CD}
 0 @>>> \BB_1 @>f_1 \oplus f_2 >> \BB_2 @>p_1( or \; p_2)>> B_3 @>>>0
 \end{CD}
 $$
 However, this complex is {\em not} exact --- $\Ker(p_1)$ is  larger than $(f_1\oplus f_2)(\BB_1)$.  The next lemma describes this difference. 
 \\
 \begin{lem}[Rabi\cite{rabi}] Let $\BF=\BF_{B^1_2 \tilde{\ominus}_B B^2_2} = \BB_2/\BB_1$. Then one has the following diagram, in which the horizontal and vertical sequences are exact. 
 $$
 \begin{CD}
 &&&&&&0\\
 &&&&&&@VVV\\
 &&&&&&C_1\\
 &&&&&&@VV\phi V\\
 0 @>>> \BB_1 @>f>> \BB_2 @>\eta>> \BF @>>>0  \\
&&&&&&@VV\bar{p}V \\
&&&&&&B_3\\
&&&&&&@VVV\\
&&&&&&0
\end{CD}
 $$
 \label{rabilemma}
\end{lem}

\begin{proof} \cite{rabihodge}, Appendix B. We repeat the proof here as that is unpublished. The horizontal sequence is exact by definition. To show the vertical sequence is exact we have to first describe be map $\phi$.
It is defined as follows. One has maps $\pi_j:B_1^{j} \longrightarrow C_1$. Consider the natural map 
$$\begin{CD}
\tilde{\phi}:C_1\oplus C_1 @>>> (B_1^1 \oplus B_1^2 )/\Delta_{A_1}@>(f_1,f_2)>>\BB_2=H_2/D_2
\end{CD}
$$
$$\tilde{\phi}(c_1,c_2) \rightarrow (\pi_1^{-1}(c),\pi_2^{-1}(c)) \rightarrow (f_1(\pi_1^{-1}(c_1)),f_2(\pi_2^{-1}(c_2)))$$
where $\Delta_{A_1}=\{(i_1(a),i_2(a))| a \in A_1\}$. $\phi$  gives a  well defined map 
$$(C_1\oplus C_1)/\Delta_{C_1} \longrightarrow \BB_2/ \tilde{\phi}(\Delta_{C_1})$$
where  $\Delta_{C_1}=\{(c,-c) | c \in C_1\}$ is the anti-diagonal.  This is well defined as if $(b_1,b_2)$ and $(b'_1,b'_2)$ are in $(\pi_1^{-1}(c_1),\pi_2^{-1}(c_2))$ we have to show 
$$(f_1(b_1),f_2(b_2)) \equiv (f_1(b'_1),f_2(b'_2)) \mod \tilde{\phi}(\Delta_{C_1})$$
or 
$$(f_1(b_1-b'_1),f_2(b_2-b'_2)) \in \tilde{\phi}(\Delta_{C_1}).$$
From exactness, we have $b_1-b'_1=i_1(a_1)$ and $b_2-b'_2=i_2(a_2)$  with $a_i \in A_1$. The image of $\Delta_{C_1}$ under $(\pi_1^{-1},\pi_2^{-1})$ consists of $(b,b')$ such that $\pi_1(b)=\pi_2(b')$. $(i_1(a_1),i_2(a_2))$ lie in this image, hence
$$(f_1(i_1(a_1)),f_2(i_2(a_2))) = (f_1(b_1-b'_1),f_2(b_2-b'_2)) \in  \tilde{\phi}(\Delta_{C_1}).$$
Note that the pre-image  $(\pi_1^{-1},\pi_2^{-1})(\Delta_{C_1})$ in $B_1^1 \oplus B_1^2)/\Delta_{A_1}$  is the Baer difference $\BB_1$. Further, $(C_1\oplus C_1)/\Delta_{C_1}) \simeq C_1$. Hence one has a  map $\phi: C_1 \rightarrow \BF=\BB_2/\BB_1$ and  we get a exact sequence 
$$0 \longrightarrow  C_1 \stackrel{\phi}{\longrightarrow} \BF_{B^1_2 \tilde{\ominus}_B B^2_2}  \stackrel{\bar{p}}{\longrightarrow} B_3 \longrightarrow 0$$
This sequence is exact as if $b=(b^1_2,b^2_2)$ is in $\BF_{B^1_2 \tilde{\ominus}_B B^2_2}$ and $\bar{p}(b)=0$, then $p_1(b^1_2)=p_2(b^2_2)=0$. So $b^1_2$ and $b^2_2$ lie in the image of  $B_1^1\oplus B_1^2$ --- say $b_2^1=f_1(b_1^1)$ and $b_2^2=f_2(b_1^2)$.  Let $c_i=\pi_1(b_1^1)$ and $c_2=\pi_2(b_1^2)$.  Then 
$$b=\phi(c_1,c_2)$$
so it lies in the image of $\phi$.

\end{proof}

In general,  for any $\ZZ$-linear combination $m \cdot B^1_2 \tilde{\ominus}_B \;n \cdot B^2_2$ of $B_2^1$ and $B_2^2$ we get an extension class ${\mathbf f}_{m\cdot  B^1_2 \tilde{\ominus}_B \;n \cdot B^2_2}$  in $\Ext(B_3,C_1)$ corresponding to $\BF_{m \cdot B^1_2 \tilde{\ominus}_B \;n \cdot B^2_2}$.  The relation between this and the extension classes constructed above is given as follows:
\begin{cor} Let ${\mathbf f}_{B_2^j}$ and ${\mathbf f}_{m \cdot B_2^1 \tilde{\ominus}_B \; n \cdot B^2_2}$ be the extensions in $\Ext(B_3,C_1)$  described  above. Then, 
$${\mathbf f}_{m \cdot B_2^1 \tilde{\ominus}_B\; n \cdot B^2_2}=m\cdot {\mathbf f}_{B_2^1} \ominus_B \; n\cdot {\mathbf f}_{B_2^2}.$$
\label{newextensions}
\end{cor}

\begin{proof} This follows from the construction of the map $\phi$. 
\end{proof}
In the next section we apply these constructions in our particular case to get the extension class we want.

\subsection{The extension $e^4_{QR,P}$}

In this section we construct an extension $e^4_{QR,P}$ in $\Ext(\otimes^3 H^1(C),H^1(C))$ which generalises the element $e^4_{Q,P} \ominus_B e^4_{R,P}$ constructed by Colombo. Recall that we have an exact sequence 
 $$E^3_{Q,P}: 0 \longrightarrow H^1(C) \longrightarrow (J_{Q,P}/J_{Q,P}^3)^* \longrightarrow \otimes^2 H^1(C) \longrightarrow 0$$ 
 and a similar one  $E^3_{R,P}$. Also, we have the  exact sequence 
 $$E_{Q,P}^4: 0 \longrightarrow (J_{Q,P}/J_{Q,P}^3)^*  \longrightarrow (J_{Q,P}/J_{Q,P}^4)^* \longrightarrow (J_{Q,P}^3/J_{Q,P}^4)^* \longrightarrow 0$$
 and a similar  $E^4_{R,P}$.  This gives us  diagrams as in Lemma \ref{rabilemma}, with $B_1^1=(J_{Q,P}/J_{Q,P}^3)^*$,  $B_1^2= (J_{R,P}/J_{R,P}^3)^* $, $B_2^1= (J_{Q,P}/J_{Q,P}^4)^* $, $B_2^2= (J_{R,P}/J_{R,P}^4)^*$ and $A_1=H^1(C)$, $B_3=\otimes^2 H^1(C)$ and finally $C_1=\otimes^2 H^1(C)$. 
%

Let $\bff_Q$, $\bff_R$ and $\bff_{QR}$ denote the classes in $\Ext(\otimes^3 H^1(C),\otimes^2 H^1(C))$ with middle terms $H^{23}_{Q,P}, H^{23}_{R,P}$ and $H^{23}_{QR,P}$  corresponding to the diagrams for $Q$, $R$ and their generalised Baer difference.  $\bff_Q$ and $\bff_R$ are the push-forwards of $m^4_{Q,P}$ and $m^4_{R,P}$ respectively. From Corollary \ref{newextensions} one has 
$$\bff_{QR}=\bff_Q-\bff_R.$$
%

\begin{lem}  $\bff_{QR}$ is $N$-torsion in $\Ext(\otimes^3 H^1(C),\otimes^2 H^1(C))$. Namely,  
$$ N\cdot H^{23}_{QR,P}=\otimes^2 H^1(C) \oplus \otimes^3 H^1(C).$$
 \end{lem}
 
 \begin{proof} Rabi \cite{rabi}, Corollary 3.3,  states that the map
 $$\Div(C) \longrightarrow \Ext(\otimes^3 H^1(C),\otimes^{2} H^1(C))$$
given by 
$$Q \longrightarrow \bff_Q$$
 factors through $\Pic(C)$. In particular, since $N(Q)-N(R)=0 \in \Pic^0(C)$ we have that $N \cdot \bff_{QR}$ corresponds to a split extension. 
  \end{proof}
 We also know from Lemma \ref{splitting} that $m^3_{QR,P}$ is $N$-torsion. Hence from Lemma \ref{rabilemma} we get an exact sequence 
 $$\begin{CD}
 0@>>> \otimes^2 H^1(C) \oplus H^1(C) @>>> N \cdot H^4_{QR,P} @>>> \otimes^3 H^1(C) \oplus \otimes^2 H^1(C) @>>> 0.
 \end{CD}$$ 
 which gives a class  in $\Ext( \otimes^3 H^1(C) \oplus  \otimes^2 H^1(C), \otimes^2 H^1(C) \oplus H^1(C))$. 
 From the K\"unneth theorem, 
 $$\Ext( \otimes^3 H^1(C) \oplus  \otimes^2 H^1(C), \otimes^2 H^1(C) \oplus H^1(C))= \prod_{ i \in \{2,3\}, j \in\{1,2\}}  \Ext(\otimes^i H^1(C),\otimes^j H^1(C)).$$
 Define
 $$e^4_{f_{QR}}=e^4_{QR,P}  \in \Ext(\otimes^3 H^1(C), H^1(C))$$
to be the projection onto that component. Note that if $C$ is hyperelliptic, this class $e^4_{QR,P}$ is precisely the class $e^4_{QR,P}=e^4_{Q,P} \ominus_B e^4_{R,P}$ constructed by Colombo. 


\subsection{Statement of the main theorem}

Armed with the class $e^4_{QR,P} \in \Ext(\otimes^3 H^1(C), H^1(C))$ we can proceed as in Colombo.

Let $\Omega$ denote the pullback of the polarisation on $J(C)$ in $H^2(J(C),1)$ to $\otimes^2 H^1(C)(1)$. There is  an injection obtained by tensoring with $\Omega$
$$J_{\Omega}=\otimes \Omega:H^1(C)(-1) \longrightarrow \otimes^3 H^1(C).$$
We first pull back the class using the map $J_{\Omega}$ to get a class in
$$J_{\Omega}^*(e^4_{QR,P}) \in \Ext(H^1(C)(-1),H^1(C)).$$
 Tensoring with $H^1(C)$ we get a class 
 $$J_{\Omega}^*(e^4_{QR,P}) \otimes H^1(C) \in  \Ext( \otimes^2 H^1(C)(-1), \otimes^2 H^1(C)).$$
 Once again pulling back using the map $\beta:\ZZ(-1) \rightarrow \otimes^2 H^1(C)$ gives us a class 
 $$\epsilon_{QR,P}^4 \in \Ext(\ZZ(-2),\otimes^2 H^1(C)) \subset \Ext(\ZZ(-2),H^2(C \times C)).$$
Our main theorem is 
 
 \begin{thm} Let $C$ be a smooth projective  curve and $P$, $Q$ and $R$ be three distinct  points. Let  $Z_{QR}=Z_{QR,P}$ be the element of the motivic cohomology group $H^3_{\M}(J(C), \ZZ(2))$ constructed above. Let $\epsilon^4_{QR,P}$ be the extension in $\Ext_{MHS}(\ZZ(-2),\wedge^2 H^1(C))$ constructed above. Then 
 $$\epsilon^4_{QR,P}= (2g_C+1) \reg_{\ZZ}(Z_{QR}) $$
 in $\Ext_{MHS}(\ZZ(-2),\wedge^2 H^1(C ))$.  
 \label{mainthm}
 \end{thm}  
 
In other words our theorem states that the regulator of a natural cycle in the motivic cohomology group of a product of curves, being thought of as an extension class  is the same as that as a natural extension of MHS coming from the fundamental group of the curve.  In fact, it is an extension of {\em pure} Hodge structures. 

\begin{rem}[Dependence on $P$] This is not so serious.    If we do not normalise $f_{QR}$ with the condition that $f_{QR}(P)=1$ then one has to add an  expression of the form  $\log(f_{QR}(P)) \int_{C} \cdot$ to the term --- and this corresponds to adding a  {\em decomposable} element of the form $(\Delta_C,\log(f_{QR}(P)))$ to our element $Z_{QR}$. 

\end{rem}

\subsection{Carlson's representatives}

The proof of the above theorem will follow by showing that they induce the same current. For that we have to understand the how an extension class induces a current. This comes from understanding the Carlson representative. In the section we once again follow Colombo \cite{colo} and adapt her arguments  to our situation.  

If $V$ is a MHS all of whose weights are negative, then the {\em Intermediate Jacobian of $V$} is defined  to be 
$$J(V) = \frac{ V_{\C}}{F^0V_{\C} \oplus V_{\ZZ}}.$$
This is a generalised torus - namely a group of the form $\C^a/\ZZ^b \simeq (\C^*)^{b} \times (\C)^{a-b}$ for some $a$ and $b$. 

An extension of mixed Hodge structures 
$$0 \longrightarrow A \stackrel{\iota}{\longrightarrow} H \stackrel{\pi}{\longrightarrow} B \longrightarrow 0$$
is called {\em separated}  if the lowest non-zero weight of $B$ is greater than the largest non-zero weight of $A$. This implies that $\Hom_{MHS}(B,A)$ has negative weights.  Carlson \cite{carl} showed that 
$$\Ext_{MHS}(B,A) \simeq J(\Hom(B,A)).$$
This is defined as follows. As an extension of {\em Abelian groups}, the extension splits. So  one has a map $r_{\ZZ}:H \rightarrow  A$ which is a retraction --- namely $r_{\ZZ}\circ \iota=id$.  Let $s_F$ be a section in $\Hom(B_{\C},H_{\C})$ preserving the Hodge filtration. Then the Carlson representative of an extension is defined to be the class of
$$r_{\ZZ} \circ s_F \in J(Hom(B,A))$$
\subsection{The Carlson representative of $\epsilon^4_{QR}$}

We now describe explicitly the Carlson representative of the extension $\epsilon^4_{QR,P}$ constructed in the previous section. This is done in a few steps, first we describe the representative  for $e^4_{QR,P}$ and then  for its various pullbacks and push forwards to obtain that for $\epsilon^4_{QR,P}$. 
We first describe  the Carlson representative of the extension 
$$e^4_{QR,P} \in \Ext_{MHS}(\otimes^3 H^1(C),H^1(C)).$$

Let $P,Q,R$ be as above. Fix a set of loops  $\alpha_1,\alpha_2,\dots ,\alpha_{2g}$ based at $P$  in $C_{Q,R}=C\backslash\{Q,R\}$ such that they give a symplectic basis  for $H_1(C)$ --- so the intersection matrix is of the form 
$$\begin{pmatrix} 0 & I \\ -I & 0 \end{pmatrix}.$$
Let $\{dx_i\}$ be  basis of $H^1(C)_{\C}$ satisfying the following conditions.
\begin{itemize}

\item The $1$-forms $dx_i$ are harmonic. 

\item  $\int_{\alpha_i} dx_j=\delta_{ij}$, where $\delta_{ij}$ is the Kronecker Delta function. 

\end{itemize}
With this choice of $\{\alpha_i\}s$ and $\{dx_i\}s$  the volume form on $H^2(C)$ can be expressed as follows. Let 
$$c(i)=\begin{cases} 1 &\text{ if } i\leq g_C\\
                                         -1 & \text{ if } i>g_C\end{cases}
                                                $$
and $\sigma(i)=i+c(i)g_C$.  The volume form is 
$$\sum_{i=1}^{2g_C} c(i) dx_i \wedge dx_{\sigma(i)}$$
and from that one gets that the Poincar\'e dual of $\alpha_i$ is $c(i)dx_{\sigma(i)}$. 
From the above description, we have that the Carlson representative of $e^4_{QR,P}$  is given by 
$$p_1 \circ r_{\ZZ} \circ s_F \circ i_3$$
where 

\begin{itemize}
\item  $p_1$ is the projection of $N \cdot  H^3_{QR,P} \simeq H^1(C) \oplus \otimes^2 H^1(C)  \stackrel{p_1}{\longrightarrow} H^1(C)$.
\item  $i_3$ is the inclusion map $\otimes^3 H^1(C) \stackrel{i_3}{\hookrightarrow} \otimes^3 H^1(C) \oplus \otimes^2 H^1(C)$. 
\end{itemize}
To describe $s_F$ we need a little more. Let $\tilde{\ominus}_B$ be the generalised Baer difference. Let 
$$s_F \circ i_3: \otimes^3 H^1(C) \longrightarrow N\cdot H^4_{QR,P} \simeq  N\cdot \left( (J_{Q,P}/J_{Q,P}^4)^*   \tilde{\ominus}_B (J_{R,P}/J_{R,P}^4)^* \right)$$
 be the  section preserving the Hodge filtration given by 
 $$s_F(dx_i \otimes dx_j \otimes dx_k )=(I^{ijk}_Q,I^{ijk}_R).$$
Here  $I^{ijk}_{\bullet} \in (J_{\bullet,P}/J_{\bullet,P}^4)^*$ for $\bullet \in \{Q,R\}$  are two iterated integrals with 
\begin{equation}
I^{ijk}_{\bullet}= N \left(\int dx_i dx_j dx_k +  dx_i \mu_{jk,\bullet}+\mu_{ij,\bullet}dx_k + \mu_{ijk,\bullet}  \right)
\label{iteratedformula}
\end{equation}
where $\mu_{ij,\bullet}$,  $\mu_{jk,\bullet}$ and $\mu_{ijk,\bullet}$ are smooth, logarithmic $(1,0)$ forms on $C_{\bullet}$ such that 
\begin{itemize}
\item $d\mu_{jk,\bullet}+dx_j \wedge dx_k=0$
\item $d\mu_{ij,\bullet}+dx_i \wedge dx_{j}=0$
\item $dx_i \wedge \mu_{jk,\bullet} + \mu_{ij,\bullet}\wedge dx_k +d\mu_{ijk,\bullet}=0.$
\end{itemize}
To compute the element of $\Hom(\otimes^3 H^1(C)_{\C}, H^1(C)_{\C})$ obtained as the projection under $p_1$,  we describe it as an element of $H_1(C)_{\C}^*=\Hom(H_1(C),\C)$.  The map from 
$$H^1(C) \longrightarrow (H^1(C) \oplus H^1(C))/\Delta_{H^1(C)}$$ 
is given by 
$$x \longrightarrow (x,-x).$$ 
Further, if $\alpha$ is a smooth loop based at $P$, the class in $H_1(C)$ corresponding to it is $1-\alpha$. 
So one has $p_1 \circ r_{\ZZ} \circ s_F \circ i_3 \in \Hom(\otimes^3 H^1(C)_{\C},H^1(C)_{\C})$ 
$$p_1 \circ r_{\ZZ} \circ s_F \circ i_3(dx_i\otimes dx_j\otimes dx_k)(\alpha)=\int_{1-\alpha} I^{ijk}_Q - \int_{1-\alpha} I^{ijk}_R$$
$$ =N \left(\int_{1-\alpha} dx_i (\mu_{jk,Q}-\mu_{jk,R})+(\mu_{ij,Q} -\mu_{ij,R})dx_k + (\mu_{ijk,Q}-\mu_{ijk,R}) \right).$$  
\begin{rem} We can choose the logarithmic forms $\mu_{ij,\bullet}$ and $\mu_{ijk,\bullet}$, for $\bullet \in \{Q,R\}$,  satisfying the following 
\begin{itemize}
\item $\mu_{ij,\bullet}=-\mu_{ji,\bullet}$.
\item For $|i-j|\neq g_{C}$, $\mu_{ij,\bullet}$ is  smooth on $C$,  as $d\mu_{ij,\bullet}=dx_j \wedge dx_i=0$. As $H^2(C_{Q,R},\ZZ)=0$ and $\mu_{ij,\bullet}$  is  smooth, it is  orthogonal to all closed forms, that is, $\mu_{ij,\bullet}\wedge dx_k$ is exact. If $dx_k$ is harmonic, then $\mu_{ij,\bullet}\wedge dx_k=0$. 
\item $\mu_{i\sigma(i),\bullet}$  has a logarithmic singularity at $\bullet$ with residue $c(i)$. 
\item $\mu_{ij,Q}-\mu_{ij,R}=0$ if $|i-j| \neq g_C$.
\item $\mu_{i\sigma(i),Q}-\mu_{i\sigma(i),R}=\frac{c(i)}{N}d\log(f)$, where  $f=f_{QR}$ is a function such that $\div(f)=NQ-NR$. We can normalise $f_{QR}$ once again by requiring that $f_{QR}(P)=1$. 
\label{properties}
\end{itemize}
\end{rem}
In terms of the basis  of harmonic forms of  $H^1(C)$,  $\Omega \in \otimes^2 H^1(C)$ is expressed as 
$$\Omega=\sum_{i=1}^{g_C} dx_i \otimes dx_{(i+g_C)} - dx_{(i+g_C)} \otimes dx_i=\sum_{i=1}^{2g_C} c(i)dx_i \otimes dx_{\sigma(i)}$$
With the choices of $\mu_{ij,\bullet}$ and $\mu_{ijk,\bullet}$  as above, we have the following theorem:
\begin{thm} Let  $G_{QR,P} \in \Hom(H^1(C)(-1)_{\C},H^1(C)_{\C})$ be a  Carlson representative  corresponding to the extension class $J_{\Omega}^*(e^4_{QR,P})$. It is given by 
$$G_{QR,P}(dx_k)(\alpha_j)=p_1 \circ r_{\ZZ} \circ s_F \circ i_3( dx_k \otimes  \Omega) (\alpha_j)=  (2g_C+1) \int_{\alpha_j}  \log(f)dx_k  -  N \int_{\alpha_j} W(dx_k)  $$
in $J(\Hom(H^1(C)(-1),H^1(C))$, where 
$$W(dx_k)=\sum_{i=1}^{2g_C} c(i)( \mu_{k i \sigma(i),Q} -\mu_{ki\sigma(i),R}).$$
\end{thm}
\begin{proof} Let $S_F$ denote the map $S_F=s_F \circ i_3 \circ J_{\Omega}:H^1(C)(-1) \rightarrow N \cdot H^4_{QR,P}$. This is given by 
$$S_F(dx_k)=\sum_{i=1}^{2g_C} c(i) s_F( dx_k \otimes dx_i \otimes dx_{\sigma(i)})$$
From \eqref{iteratedformula} one has 
$$S_F(dx_k)=\left( \sum_{i=1}^{2g_C} c(i) \int I_Q^{ki\sigma(i) },\sum_{i=1}^{2g_C} c(i) \int  I_R^{ki\sigma(i)} \right)$$
Evaluating on a path $\alpha_j$ based at $P$ using the maps described above, this is 
$$\sum_{i=1}^{2g_C} \int_{1-\alpha_j} c(i)   \left( I_Q^{ki\sigma(i)} -  I_R^{k i \sigma(i)} \right)$$
From  Remark \ref{properties}, the leading terms and several of the lower order terms cancel out and
$$ \mu_{ki,Q}-\mu_{ki,R}=c(k) \delta_{k\sigma(i)}  d\log(f)/N $$
 and finally 
 $$ \mu_{i\sigma(i),Q}-\mu_{i\sigma(i),R}=c(i)d\log(f)/N.$$
 Since $c(i)^2=1$ what remains is 
$$   \sum_{i=1}^{2g_C}  \int_{1-\alpha_j} dx_k \frac{ df}{f}  -  \int_{1-\alpha_j} \frac{df}{f} dx_k  +  N \sum_{i=1}^{2g_C} c(i) \int_{1-\alpha_j} \left(\mu_{k i \sigma(i),Q}-\mu_{k i \sigma(i),R} \right).$$
Let 
$$W(dx_k)=  \sum_{i=1}^{2g_C} c(i)   \left(\mu_{k i \sigma(i),Q}-\mu_{k i \sigma(i),R} \right).$$
Recall that $\gamma=f^{-1}([0,\infty])$.  On $C\backslash \gamma$, $d\log(f)$ is exact. So if $\alpha_j \cap \gamma=\emptyset$ then we can evaluate the integral using Lemma \ref{basicproperties}(3). If $\alpha_j \cap  \gamma \neq \emptyset$, one has to do the computation on a path lifting of $\alpha_j$ on a covering of $C$ where $d\log(f)$ is exact. The difference in the two integrals is given by a multiple of $2\pi i \int_{\alpha_j} dx_k$ -- hence is in $\Hom_{\ZZ}(H^1(C)(-1), H^1(C))$ -- which is $0$ in the intermediate Jacobian. 

Hence we have, using Lemma \ref{basicproperties}(3) and the fact that we have chosen $f$ with $f(P)=1$,
$$\int_{1-\alpha_j}  dx_k\frac{ df}{f} =  - \int_{1-\alpha_j} \log(f)dx_k$$
and 
$$\int_{1-\alpha_j} \frac{df}{f}dx_k= +  \int_{1-\alpha_j} \log(f)dx_k$$
Since integration over a point , which corresponds to the loop 1,  is zero. Hence the integral is 
$$G_{QR,P}(dx_k)(\alpha_j)=    -(2g_C+1) \int_{1-\alpha_j} \log(f)dx_k + N \int_{1-\alpha_j} W(dx_k).$$ 
$$=(2g_C+1) \int_{\alpha_j}  \frac{df}{f} dx_k - N  \int_{\alpha_j}W(dx_k).  $$
\end{proof}

\begin{rem} It is convenient to have the iterated integral expression for the Carlson representative as well, so we note it here 
$$G_{QR,P}(dx_k) (\alpha_j)=(2g_C+1) \int_{\alpha_j}  \frac{df}{f} dx_k - N  \int_{\alpha_j}W(dx_k).  $$
\end{rem}

We have computed the Carlson representative $G_{QR,P}$ of our class  in $\Ext(H^1(C)(-1),H^1(C))$.  We now tensor with $H^1(C)$ and pull back using the map $\otimes \Omega:\ZZ(-1)\longrightarrow \otimes^2 H^1(C)$.  This gives us an element  of $\Ext(\ZZ(-2), \otimes^2 H^1(C))$. We denote its Carlson representative by $F_{QR,P}$.                                                  
\\
\begin{lem} The Carlson representative of the class in $\Ext(\ZZ(-2),\otimes^2 H^1(C))$ is given by 
\label{fqrplemma}
$$F_{QR,P} = (G_{QR,P} \otimes Id) \circ \otimes \Omega$$ 
in $(\otimes^2 H^1(C)_{\C})^*$. On an element $\alpha_j \otimes \alpha_k$ 
 it is given by
\begin{equation}
F_{QR,P}(\Omega)(\alpha_j \otimes \alpha_k)=c(\sigma(k))\left((2g_C+1) \int_{\alpha_j} \log(f)  dx_{\sigma(k)} - N \int_{\alpha_j} W(dx_{\sigma(k)})\right)
\label{fqrp}
\end{equation}
\end{lem}

\begin{proof} Recall that 
$$\Omega=\sum_i^{2g_C} c(i)dx_i \otimes dx_{\sigma(i)}.$$
From above we have 
$$(G_{QR,P} \otimes Id )(\Omega)(\alpha_j \otimes \alpha_k) = \sum_i^{2g_C}  c(i) G_{QR,P}(dx_i)(\alpha_j) \cdot  Id(dx_{\sigma(i)}) (\alpha_k).$$
From the choice of $\alpha_k$ one has 
$$Id(dx_{\sigma(i)}) (\alpha_k) =\delta_{k\sigma(i)}.$$
Hence, in the sum above,  precisely one term survives -- when $i=\sigma(k)$, and we have 
$$(G_{QR,P} \otimes Id ) (\Omega)(\alpha_j \otimes \alpha_k)=c( \sigma(k)) G_{QR,P}(dx_{\sigma(k)})(\alpha_j).$$
In particular
\begin{align*}
F_{QR,P}(\Omega)(\alpha_j \otimes \alpha_k)=&c(\sigma(k)) G_{QR,P}(dx_{\sigma(k)})(\alpha_j)\\
=&c(\sigma(k))\left((2g_C+1) \int_{\alpha_j} \log(f)  dx_{\sigma(k)} - N \int_{\alpha_j} W(dx_{\sigma(k)})\right).
\end{align*}

\end{proof}

We now recall a  lemma due to Colombo  which relates integrals over the curve $C- \gamma$ with integrals over paths. This is crucial in relating the two expressions for the regulator. 
\\
\begin{lem}[Colombo]\cite{colo} Let $\gamma$ be the path $f^{-1}([0,\infty])$. Let $\alpha$ be a smooth, simple loop on $C$ transverse to $\gamma$. Let $\phi$, $\psi$ and $\omega$ be three smooth $1$-forms on $C$ such that $\phi$, $\psi$ and $\Theta=(\log(f)\psi+\omega)$ are closed and $\phi$ is the Poincar\'{e} dual of the class of $\alpha$. Then 
$$\int_{\alpha} \Theta=\int_{C-\gamma} \phi \wedge \Theta + 2\pi i \int_{\gamma} \phi\psi$$
\label{colombolemma}
\end{lem}
\begin{proof}  This is Colombo's Proposition 3.3.
\end{proof}

We now apply this in the case of interest to us.

\begin{cor} Choose $\alpha_j$ to be simple closed loops                                                                                                                                                                                                                                                                                                                                                                                                                                                                                                                                                                                                                                                                                                                                                                                                                                                                                                                                                                                                                                                                                                                                                                                                                                                                                                                                                                                                                                                                                                                                                                                                                                                                                                                                                                                                                                                                                                                                                                                                                                                                                                                                                                                                                                                                                                                                                                                                                                                                                                                                                                                                                                                                                                                                                                                                                                                                                                                                                                                                                                                                                                                                                                                                                                                                                                                                                                                                                                                                                                                                                                                                                                                                                                                                                                                                                                                                                                                                                                                                                                                                                                                                                                                                                                                                                                                                                                                                                                                                                                                                                                                                                                                                                                                                                                                                                                                                                                                                                                                                                                                                                                                                                                                                                                                                                                                                                                                                                                                                                                                                                                                                                                                                                                                                                                                                                                                                                                                                                                                                                                                                                                                                                                                                                                                                                                                                                                                                                                                                                                                                                                                                                                                                                                                                                                                                                                                                                                                                                                                                                                                                                                                                                                                                                                                                                                                                                                                                                                                                                                                                                                                                                                                                                                                                                                         transverse to $\gamma$. Then we have 
$$F_{QR,P}(\Omega)(\alpha_j \otimes \alpha_k)=(2g_C+1)c(j)c(\sigma(k)) \left(\int_{C-\gamma}  dx_{\sigma(j)} \wedge \left( \log(f)dx_{\sigma(k)}-\frac{N}{(2g_C+1)}W(dx_{\sigma(k)})\right)  + 2\pi i \int_{\gamma}  dx_{\sigma(j)} dx_{\sigma(k)} \right). $$
\end{cor}
\begin{proof} One has $c(j) dx_{\sigma(j)}$ is the Poincar\'{e} dual of $\alpha_j$. Hence we can apply the above lemma with
\begin{itemize}
\item  $\phi=c(j)dx_{\sigma(j)}.$
\item  $\psi=dx_{\sigma(k)}.$ 
\item  $\Theta=\log(f)dx_{\sigma(k)}-\frac{N}{(2g_c+1)}W(dx_{\sigma(k)}).$
\end{itemize}

$\Theta$ is closed follows from Remark \ref{properties}.  From  \eqref{fqrp} we have 
$$\frac{1}{(2g_C+1)} F_{QR,P}(\Omega)(\alpha_j \otimes \alpha_k)=c(\sigma(k))\left(\int_{\alpha_j}  \log(f)dx_{\sigma(k)}-\frac{N}{(2g_C+1)}W(dx_{\sigma(k)})\right)$$
$$=c(\sigma(k))c(j)\left(\int_{C-\gamma} dx_{\sigma(j)} \wedge \left( \log(f)dx_{\sigma(k)}-\frac{N}{(2g_C+1)}W(dx_{\sigma(k)})\right) + 2\pi i \int_{\gamma} dx_{\sigma(j)} dx_{\sigma(k)}\right).$$
Hence 
$$F_{QR,P}(\Omega)(\alpha_j \otimes \alpha_k)=(2g_C+1)c(j)c(\sigma(k)) \left(\int_{C-\gamma}  dx_{\sigma(j)} \wedge \left( \log(f)dx_{\sigma(k)}-\frac{N}{(2g_C+1)}W(dx_{\sigma(k)})\right) + 2\pi i \int_{\gamma}  dx_{\sigma(j)} dx_{\sigma(k)} \right) $$
\end{proof}  
$F_{QR,P}(\Omega)$ determines an element of the intermediate Jacobian of $(\otimes^2 H^1(C))^*$
$$J(\otimes^2 H^1(C)^*) \simeq \frac{F^1(\otimes^2 H^1(C)_{\C}^*)}{(\otimes^2 H^1(C))^*}$$
so to determine $F_{QR,P}(\Omega)$ it suffices to evaluate it on elements of $F^1(\otimes^2 H^1(C,\C))^*$.   We can choose the basis $dz_i$ of the space of holomorphic $1$-forms  such that 
$$\int_{\alpha_i} dz_j=\delta_{ij} \hspace{1in} 1\leq i\leq g$$ 
 where $\{\alpha_i\}$ is the symplectic basis. Since $c(j)dx_{\sigma(j)}$ is dual to $\alpha_j$,
 $$dz_j=dx_{j} + \sum_{i=1}^{g} A_{ji}dx_{i+g} \hspace{1in} \text { where } A_{ji}=\int_{\alpha_{i+g}} dz_j$$ 
Let $\zeta_{j}=c(\sigma(j))\alpha_{\sigma(j)}+\sum_{1\leq i\leq g}A_{ji}c(i)\alpha_{i},$ where $j\leq g$. Then 
\begin{prop}[Colombo, \cite{colo}, Prop 3.4] The map $F_{QR,P}(\Omega)$ evaluated on elements of the form $\zeta_i \otimes \alpha_j$ is 
 $$F_{QR,P}(\Omega)(\zeta_i \otimes c(\sigma(j))\alpha_{\sigma(j)})=(2g_C+1)\left(\int_{C-\gamma}  \log(f) dz_i \wedge  dx_j  + 2\pi i \int_{\gamma}  dz_i dx_j \right) $$
In other words 
$$dz_i \wedge W(dx_j)=0.$$
\end{prop}

\begin{proof} $dz_i$ and $W(dx_j)$ are both $(1,0)$ forms. Hence their wedge product is a $(2,0)$ form and is therefore $0$.

\end{proof}

In fact, the theorem holds for the other term as well. 

\begin{prop} For a suitable choice of $\mu_{ijk,Q}$ and $\mu_{ijk,R}$ one has 
\begin{align*}W(dz_{i}):=&W(dx_{i})+\sum_{k}A_{ki}W(dx_{i+g})=0 
\end{align*}
\label{wdz}
\end{prop}
\begin{proof} \cite{colo} Lemma $3.1$.
\end{proof}
Hence we have 
\begin{thm} 
$$F_{QR,P}(\Omega)(c(\sigma(j))\alpha_{\sigma(j)} \otimes \zeta_i)=(2g_C+1)\left(\int_{C-\gamma}  \log(f) dx_j \wedge  dz_i + 2\pi i \int_{\gamma}  dx_j dz_i \right).$$
\end{thm}
%
Comparing this with the regulator term in Theorem \ref{regform}  we get 

\begin{thm} Let $Z_{QR}$ be the motivic cohomology cycle constructed above  and $\epsilon_{QR,P}^4$ the extension in $\Ext_{MHS}(\ZZ(-2),\wedge^2 H^1(C))$. We use $\epsilon^4_{QR,P}$ to denote its Carlson representative as well.  Then one has 
$$ \epsilon^4_{QR,P}(\omega)=(2g_C+1)\reg_{\ZZ}(Z_{QR})(\omega)$$
where $\omega \in F^1\wedge^2 H^1(C)$. 
\label{mainthm}
\end{thm}
\begin{proof} It suffices to check this on $dz_i \wedge dx_j=dz_i \otimes dx_j-dx_j \otimes dz_i$. The result then follows by comparing the formula for the  Carlson representative $F_{QR,P}$ in Lemma \ref{fqrplemma} with the expression for the regulator in Theorem  \ref{regform} using Lemma \ref{colombolemma}

From Theorem 4.17 and Lemma \ref{fqrplemma} we have
\begin{align*}
F_{QR,P}(\Omega)(c(\sigma(j))\alpha_{\sigma(j)} \otimes \zeta_i)=&(2g_C+1)\left(\int_{C-\gamma}  \log(f) dx_{j}  \wedge dz_i    + 2\pi i \int_{\gamma}  dx_j d z_j \right)\\
=&(2g_{C}+1)c(\sigma(j))\left( \int_{\alpha_{\sigma(j)}} \log(f)  d z_i - N \int_{\alpha_{\sigma(j)}} W(dz_i)\right)
\end{align*}
Applying Proposition \ref{wdz} we get 
$$F_{QR,P}(\Omega)(c(\sigma(j))\alpha_{\sigma(j)} \otimes \zeta_i)=(2g_{C}+1)c(\sigma(j)) \int_{\alpha_{\sigma(j)}} \log(f)  d z_i. $$
For the other part, from Proposition 4.15 one has 
\begin{align*}
F_{QR,P}(\Omega)(\zeta_i \otimes c(\sigma(j))\alpha_{\sigma(j)})&=(2g_C+1)\left(\int_{C-\gamma}  \log(f) dz_i \wedge  dx_j  + 2\pi i \int_{\gamma}  dz_i dx_j \right)\\
&=-(2g_C+1)\left(\int_{C-\gamma}  \log(f) dx_j \wedge  dz_i  - 2\pi i \int_{\gamma}  dz_idx_j \right)
\end{align*}
%
%
Using Lemma \ref{basicproperties} $(2)$ and Stokes' theorem we get 
$$\int_{\gamma}  dz_i d x_j+ \int_{\gamma}  dx_j d z_j=\int_{\gamma} dz_i \int_{\gamma} dx_j=0.$$
as $\gamma=\partial D$ is exact  and the forms are closed. Hence $\int_{\gamma} dx_jdz_i=-\int_{\gamma} dz_idx_j$
Therefore 
\begin{align*}
F_{QR,P}(\Omega)(\zeta_i \otimes c(\sigma(j))\alpha_{\sigma(j)})&=-(2g_C+1)\left(\int_{C-\gamma}  \log(f) dx_j \wedge  dz_i  + 2\pi i \int_{\gamma}  dx_jdz_i \right)\\
&=-F_{QR,P}(\Omega)( c(\sigma(j))\alpha_{\sigma(j)} \otimes \zeta_i)
\end{align*}

Hence we get 
\begin{align*}
F_{QR,P}(\Omega)(c(\sigma(j))\alpha_{\sigma(j)} \wedge \zeta_i)=&F_{QR,P}(\Omega)(c(\sigma(j))\alpha_{\sigma(j)} \otimes \zeta_i) - F_{QR,P}(\Omega)(\zeta_i \otimes c(\sigma(j))\alpha_{\alpha_{\sigma(j)}})\\
=& 2 F_{QR,P}(\Omega)(c(\sigma(j))\alpha_{\sigma(j)} \otimes \zeta_i)\\
=& 2(2g_C+1)\left(\int_{C-\gamma}  \log(f) dx_{j}  \wedge dz_i    + 2\pi i \int_{\gamma} dx_jdz_i\right) \\
=& 2(2g_C+1)c(\sigma(j)) \int_{\alpha_{\sigma(j)}} \log(f)  d z_i
\end{align*}
On the other hand, from Theorem \ref{regform}
$$(2g_C+1)\reg_{\ZZ}(Z_{QR})( dx_j\wedge dz_i) =2(2g_C+1)\left(\int_{C-\gamma}  \log(f) dx_{j}  \wedge dz_i    + 2\pi i \int_{\gamma} dx_jdz_i\right)$$
\end{proof}

Recall that we have assumed in both cases  that $f_{QR}(P)=1$. If we do not make that assumption, then one has a term corresponding to a decomposable element that one has to account for. However, if we work modulo the decomposable cycles we can ignore that term.

\begin{cor}

Let $Z_{QR,P}$ be the element of $H^{2g-1}_{\M}(\J(C),\Q(g))$ and let $\eta$ and $\omega$  be two closed $1$ forms on $C$ with $\omega$ holomorphic. Let $\alpha$ be the  Poincar\'e dual of $\eta$. Then
$$\reg_{\Q}(Z_{QR,P})(\eta \wedge \omega) =2 \int_{\alpha} \log(f_{QR})\omega$$

\end{cor}
\begin{proof} From \eqref{fqrp} we have an integral expression for the Carlson representative of the extension class  which by  Theorem \ref{mainthm} is the regulator of  the cycle $Z_{QR,P}$ when computed against $dx_j \wedge dz_k$. This shows 
$$\reg_{\ZZ}(Z_{QR,P})(dx_{j}\wedge dz_{k} )=2c(\sigma(j))\int_{\alpha_{\sigma(j)}}  \log(f)dz_k.$$
Here $c(j)\alpha_{\sigma(j)}$ is the Poincar\'e dual of $dx_j$. The  integral expression on right hand side is the value of a functional evaluated on the form $dx_j \wedge dz_i$ which has to be considered modulo the lattice $H_2(\J(C),\Q)$.  This expression can be extended linearly to $\eta \wedge \omega$ to give the final expression. 
\end{proof}

\section{Remarks on Motives}

There are various candidates for the category of mixed motives \cite{levi}  --- Voevodsky and Huber have candidates for the triangulated category of mixed motives and Nori and Deligne-Jannsen have candidates for the Abelian category itself.   Cushman \cite{cush} showed  that Nori's motives can be used to get a motivic structure on the group ring of the fundamental group --- so one expects that the same sequences we use would give extensions of mixed motives in Nori's category. 

An alternative to Cushman's way of constructing Nori motives for the fundamental group was suggested to us by N. Fakhruddin.  Nori's category requires a realisation in terms of relative cohomology groups.  In the case of the fundamental groups this is given in the paper of  Deligne and Goncharov \cite{dego} Section 3, (Proposition 3.4).  

If $X$ is a smooth variety and $P$ a distinguished point, they show that  the Hodge structure on the graded pieces of the group ring of the fundamental group can be realised  as the Hodge structure on the  relative cohomology groups of pairs $(X^s, \cup_{i=0}^{s} X_i )$, 
where 
\begin{itemize}

\item $X^s=X \times \dots \times X$ $s$-times
\item $X_0$ is the sub-variety given by $t_1=P$ --- namely $\{P\} \times X^{s-1}$
\item $X_i$ is the sub-variety given by $t_i=t_{i+1}$ for $0 < i <s$ --- namely $X^{i-1}\times \Delta \times X^{s-(i+1)}$, where $\Delta$ is the diagonal in $X \times X$ in the $i^{th}$ and $(i+1)^{st}$ places.  
\item $X_s$ is given by $t_s=P$ --- namely $X^{s-1} \times \{P\}$. 

\end{itemize}
We have 
$$H^s(X^s /\cup_{i=0}^{s} X_i,\C) \simeq \Hom(J_P/J_P^{s+1},\C)=H^0(\bar{B}_s(X),P)$$
For example, when $s=1$ we have 
$$H^1(X/ \{P\},\C) \simeq H^1(X,\C) \simeq \Hom(J_P/J_P^2,\C).$$
Hence the motive underlying the Hodge structure on the graded quotient of the   group ring of the fundamental group $J_P/J_P^s$   is the motive associated to the pair $(X^s,\cup_{i=0}^{s} X_i)$. Namely, to this object one can associate a de Rham, \'{e}tale and Betti realization which are isomorphic when the field of coefficients is large enough. 

In particular, our constructions above work in the category of Nori motives.  In the special case when $C$ is a modular curve $X_0(N)$, this shows that the element $Z_{\Delta_{N}}$  in the motivic cohomology group constructed by Beilinson can be thought of as extension of {\em motives}  coming from the fundamental group.  Kings \cite{king}  showed this in the case of $H^2_{\M}(X_0(N),\ZZ(2))$ for Huber's motives.

\subsection{Remarks on Degenerations}

Collino shows that his cycle can be viewed as a `degeneration' of the Ceresa cycle. We expect that the Bloch-Beilinson cycle too can be viewed as a suitable degeneration of the modified diagonal cycle. In a recent preprint \cite{iymu}, Iyer and M\"uller-Stach have worked out a special case of this.  A sketch of the idea from the point of view of extensions is as follows:

 Suppose $(C,P)$ is a pointed curve  and assume that it degenerates to a nodal curve $(C_{\Sigma},P)$ with a node at $\Sigma \neq P$. Let $\pi:D \rightarrow C_{\Sigma}$ denote the normalizaition and let $\pi^{-1}(\Sigma)=\{Q,R\}$. Further assume that $N(Q-R)$ is $0$ in $J(D)$.  Then one has an exact sequence of mixed Hodge structures coming from the Mayer-Vietoris sequence 
 $$0  \longrightarrow \bar{H}^0(\{Q,R\}) \stackrel{\iota}{\longrightarrow} H^1(C_{\Sigma}) \stackrel{\pi^*}{\longrightarrow} H^1(D) \longrightarrow 0$$
 where $\bar{H}^0(\{Q,R\})$ denotes the reduced cohomology and $\iota$ denotes the map which takes the path from $Q$ to $R$ in $D$ to a loop in $C_{\Sigma}$. 
 $\bar{H}^0(\{Q,R\}) \simeq \ZZ$, hence this gives an extension class in $\Ext_{MHS}(H^1(D), \ZZ)$. Carlson \cite{carl} shows that this groups is
 $$\Ext_{MHS} (H^1(D),\ZZ) \simeq  \Ext_{MHS} (\ZZ(-1),H^1(D)) \simeq J(D)$$
The class in $J(D)$  is the class of $Q-R$ and by our assumption, $N(Q-R)=0$ in $J(D)$, so the extension splits, namely
 $$N\cdot H^1(C_{\Sigma}) \simeq \ZZ \oplus H^1(D)$$
The modified diagonal cycle $Z_P$ gives a class in $\Ext_{MHS}(\ZZ(-2),\otimes^3 H^1(C))$ which degenerates to give a class in $\Ext_{MHS}(\ZZ(-2),\otimes^3 H^1(C_{\Sigma}))$. 
$$\otimes^3 (N\cdot H^1(C_{\Sigma})) \simeq \otimes^3  (\ZZ \oplus H^1(D)) \simeq \bigoplus_{m+n=3} \otimes^m \ZZ \otimes^n H^1(D).$$
In particular, it has  sub Hodge structures isomorphic to  $\otimes^2 H^1(D)$. Projecting the class of the modified diagonal cycle $Z_P$ in $\Ext(\ZZ(-2), \otimes^3 H^1(C_{\Sigma}))$ onto those components gives classes in $\Ext_{MHS}(\ZZ(-2),\otimes^2 H^1(D))$ and we expect the following to be true: 

\begin{conj} The modified diagonal cycle in  $Z_P \in CH^2_{\hom}(C^3)$ degenerates to (a non-zero multiple of) the cycle $Z_{QR;P}$ in $H^3_{\M}(D \times D,\Q(2))$. 
\end{conj}

Iyer and M\"uller-Stach  \cite{iymu} show this is in the case when $C$ is a genus three curve hence $D$ is genus 2 hence hyperelliptic. However, they use the fact that the modified diagonal cycle on $D^3$ is torsion, but we expect the result to hold regardless of that. 

At the level of regulator or the Abel-Jacobi map, what happens is that the holomorphic differentials get replaced by logarithmic forms of the form $d\log(f)$, where $f$ is the function whose divisor is a multiple of $Q-R$.   A curious special case is the case of modular curves. Here, if one looks at the regulators - the regulator of the modified diagonal cycle can be expressed as an iterated integral of two cusp forms over the dual of a third. As one degenerates, the regulator of the Beilinson cycle is an iterated integral of a cusp form and an Eisenstein series over the dual of a cusp form. Degenerating further, one has that the regulator of some elements of $K_2$ of a modular curve can be expressed as the iterated integral of two Eisenstein series over the dual of a cusp form. Finally, one expects that there should be an expression for the special value of the $\zeta$-function of the field of definition of a cusp corresponding to $K_3$ as an integral of two Eisenstein series over the dual of a third Eisenstein series.

\bibliographystyle{alpha}
\bibliography{Extensions.bib}

\begin {verse}
Subham Sarkar and  Ramesh Sreekantan \\
Indian Statistical Institute, Bangalore\\
$8^{th}$ mile, Mysore Road, Bangalore 560 059, Karnataka, India\\
Email: rameshsreekantan@gmail.com, subham.sarkar13@gmail.com

\end{verse}
 
\end{document}